\documentclass{amsart}
\usepackage{amssymb,latexsym}
\usepackage{amscd,amsthm}

\newtheorem{theorem}{Theorem}[section]
\newtheorem{lemma}[theorem]{Lemma}
\newtheorem{proposition}[theorem]{Proposition}
\newtheorem{corollary}[theorem]{Corollary}

\theoremstyle{definition}
\newtheorem{definition}[theorem]{Definition}

\DeclareMathOperator{\Ext}{Ext}
\DeclareMathOperator{\Hom}{Hom}

\DeclareMathOperator{\im}{Im}

\newcommand{\cat}[1]{\mathcal{#1}}           

\newcommand{\tensor}{\otimes}

\newcommand{\class}[1]{\mathcal{#1}}   

\newcommand{\Z}{\mathbb{Z}}
\newcommand{\Q}{\mathbb{Q/Z}}

\newcommand{\ch}{\textnormal{Ch}(R)}

\newcommand{\tilclass}[1]{\widetilde{\class{#1}}}

\newcommand{\rightperp}[1]{#1^{\perp}}
\newcommand{\leftperp}[1]{{}^\perp #1}

\begin{document}

\title{The flat stable module category of a coherent ring}

\author{James Gillespie}
\address{Ramapo College of New Jersey \\
         School of Theoretical and Applied Science \\
         505 Ramapo Valley Road \\
         Mahwah, NJ 07430}
\email[Jim Gillespie]{jgillesp@ramapo.edu}
\urladdr{http://pages.ramapo.edu/~jgillesp/}
\thanks{The author wishes to express thanks to the referee for giving his or her valuable feedback.}

\keywords{abelian model structure, Gorenstein flat, stable module category}
\date{\today}

\begin{abstract}
Let $R$ by a right coherent ring and $R$-Mod denote the category of left $R$-modules. We show that there is an abelian model structure on $R$-Mod whose cofibrant objects are precisely the Gorenstein flat modules. Employing a new method for constructing model structures, the key step is to show that a module is flat and cotorsion if and only if it is Gorenstein flat and Gorenstein cotorsion.
\end{abstract}

\maketitle
 
\section{introduction}

Recall that a model structure on a category is a formal (categorical) way of introducing a homotopy theory on that category. If the category we start with is abelian then the resulting homotopy theory captures some variety of homological algebra, or relative homological algebra. In fact,  the homotopy category of a hereditary abelian model category is known to be an algebraic triangulated category. So a good way to both construct and model an algebraic triangulated category is to construct a hereditary abelian model structure. For a recent survey, see~\cite{gillespie-hereditary-abelian-models}.


The first purpose of this paper is to construct a new hereditary abelian model category structure, the \textbf{Gorenstein flat model structure}, on the category of left $R$-modules where $R$ is any right coherent ring. The cofibrant objects in this model structure are precisely the Gorenstein flat modules which were introduced and studied by Enochs and several coauthors and subsequently studied by many other authors. In particular, see~\cite{enochs-Xu-Goren flat covers, enochs-jenda-book, enochs-Goren flat covers, holm-Goren, Yang-Liang-Goren flat precovers}, but there is certainly more literature on the subject. A second purpose of this paper is to illustrate a new method from~\cite{gillespie-hovey triples} for constructing abelian model structures. This powerful method allows us to realize the existence of hereditary abelian model structures even before understanding the class of trivial objects. This seems strange; yet with this approach, we still obtain at once that the full subcategory of cofibrant-fibrant objects will form a Frobenius category whose stable category is canonically equivalent to the homotopy category of the constructed model structure. In the current case of the Gorenstein flat model structure, this full subcategory of cofibrant-fibrant objects is exactly the class of all left $R$-modules that are both Gorenstein flat and cotorsion. This full subcategory becomes a Frobenius category whose conflations, or short exact sequences, are the usual short exact sequences whose all three terms are Gorenstein flat and cotorsion. The injective-projective objects turn out to be the flat cotorsion modules. So the homotopy category of the Gorenstein flat model structure is triangle equivalent to the stable category of this Frobenius category (see Corollary~\ref{cor-Frobenius}). This generalizes the well known construction of the stable module category of a quasi-Frobenius ring $R$. Indeed when $R$ is quasi-Frobenius, it turns out that every $R$-module is both Gorenstein flat and cotorsion, recovering the standard fact that $R$-Mod is a Frobenius category in this case. Moreover, an $R$-module is flat if and only if it is injective if and only if it is projective in this case. It follows that the above construction recovers  with the usual stable module category of a quasi-Frobenius ring $R$.

Over the years different authors have given various generalizations of the stable module category of a quasi-Frobenius ring. Using the abelian model category approach, Hovey extended the notion to all Gorenstein rings in~\cite{hovey-cotorsion}, and the current author generalized this to the Ding-Chen rings in~\cite{gillespie-ding}. We show in Corollary~\ref{corollary} that the Gorenstein flat model structure agrees with these constructions. In fact, it was recently shown in~\cite{estrada-gillespie-coherent schemes} that the Gorenstein flat model structure recovers what was called the ``projective stable module category of $R$'' in the recent~\cite{bravo-gillespie-hovey}. We comment more on this in the last paragraph of the paper.



This paper is brief for a few reasons. One reason is that the method of~\cite{gillespie-hovey triples} used to construct the model structure is doing much of the work. We explain and give a precise statement of this method in the next section. When pairing this theorem with Hovey's one-to-one correspondence between cotorsion pairs and abelian model structures the problem reduces to showing that the flat modules and Gorenstein flat modules are each the left half of a complete hereditary cotorsion pair. But this is already known from the work of Enochs and coauthors. So the essential step is to show that the \emph{cores} of these cotorsion pairs coincide. All of the terminology in this paragraph will now be explained in the next section.

\section{preliminaries}\label{sec-prelims}

We let $R$ be a ring and $R$-Mod denote the category of left $R$-modules. We assume the ring has an identity and that the modules are unital.

\subsection{Cotorsion pairs and Hovey triples} For a class $\class{S}$ of $R$-modules, we let $\rightperp{\class{S}}$ denote the class of all modules $N$ such that $\Ext_R^1(S,N) = 0$ for all $S \in \class{S}$. On the other hand, $\leftperp{\class{S}}$ denotes the class of all modules $M$ such that $\Ext_R^1(M,S) = 0$ for all $S \in \class{S}$. By a \textbf{cotorsion pair} we mean a pair of classes of $R$-modules $(\class{A},\class{B})$ such that $\class{B} = \rightperp{\class{A}}$ and $\class{A} = \leftperp{\class{B}}$. We call the cotorsion pair \textbf{hereditary} if the left class $\class{A}$ is \emph{resolving} in the sense that whenever
$$ 0 \xrightarrow{} A' \xrightarrow{} A \xrightarrow{} A'' \xrightarrow{} 0 $$ is a short exact sequence with $A,A'' \in \class{A}$, then $A' \in \class{A}$ too. Since $R$-Mod has enough projectives and injectives this condition has been shown to be equivalent to the dual statement that $\class{B}$ is \emph{coresolving}. See~\cite{garcia-rozas}.

A cotorsion pair is called \textbf{complete} if it has enough injectives and enough projectives. This means that for each $R$-module $M$ there exist short exact sequences $0 \xrightarrow{} B \xrightarrow{} A \xrightarrow{} M \xrightarrow{} 0$ and $0 \xrightarrow{} M \xrightarrow{} B' \xrightarrow{} A' \xrightarrow{} 0$ with $A,A' \in \class{A}$ and $B,B' \in \class{B}$.
 The books~\cite{enochs-jenda-book} and~\cite{trlifaj-book} are standard references for cotorsion pairs.

By the main theorem of~\cite{hovey-cotorsion} we know that an abelian model structure on $R$-Mod, in fact on any abelian category, is equivalent to a triple $(\class{Q},\class{W},\class{R})$ of classes of objects for which $\class{W}$ is thick and $(\class{Q} \cap \class{W},\class{R})$ and $(\class{Q},\class{W} \cap \class{R})$ are each complete cotorsion pairs. By \textbf{thick} we mean that the class $\class{W}$ is closed under retracts (i.e., direct summands) and satisfies that whenever two out of three terms in a short exact sequence are in $\class{W}$ then so is the third. In this case, $\class{Q}$ is precisely the class of cofibrant objects of the model structure, $\class{R}$ are precisely the fibrant objects, and $\class{W}$ is the class of trivial objects. We hence denote an abelian model structure $\class{M}$ as a triple $\class{M} = (\class{Q},\class{W},\class{R})$ and for short
we will denote the two associated cotorsion pairs above by $(\tilclass{Q},\class{R})$ and $(\class{Q},\tilclass{R})$ where $\tilclass{Q} = \class{Q} \cap \class{W}$ is the class of trivially cofibrant objects and $\tilclass{R} = \class{W} \cap \class{R}$ is the class of trivially fibrant objects. We say that $\class{M}$ is \textbf{hereditary} if both of these associated cotorsion pairs are hereditary. We will also call any abelian model structure $\class{M} = (\class{Q},\class{W},\class{R})$ a \textbf{Hovey triple}.
Besides~\cite{hovey-cotorsion}, the book~\cite{hovey-model} is a standard reference for the theory of model categories.

By the \textbf{core} of an abelian model structure $\class{M} = (\class{Q},\class{W},\class{R})$ we mean the class $\class{Q} \cap \class{W} \cap \class{R}$. This notion comes up in the following theorem giving a sort of converse to Hovey's main theorem in the case that we have hereditary cotorsion pairs.

\begin{theorem}[How to construct a Hovey triple from two cotorsion pairs]\label{them-how to construct hovey triples}
Let $(\class{Q}, \tilclass{R})$ and $(\tilclass{Q}, \class{R})$ be two complete hereditary cotorsion pairs in an abelian category $\cat{C}$ satisfying the two conditions below.
\begin{enumerate}
\item $\tilclass{R} \subseteq \class{R}$ and $\tilclass{Q} \subseteq \class{Q}$.
\item $\tilclass{Q} \cap \class{R} = \class{Q} \cap \tilclass{R}$.
\end{enumerate}
Then $(\class{Q},\class{W},\class{R})$ is a Hovey triple where the thick class $\class{W}$ can be described in the two following ways:
\begin{align*}
   \class{W}  &= \{\, X \in \class{C} \, | \, \exists \, \text{a short exact sequence } \, X \rightarrowtail R \twoheadrightarrow Q \, \text{ with} \, R \in \tilclass{R} \, , Q \in \tilclass{Q} \,\} \\
           &= \{\, X \in \class{C} \, | \, \exists \, \text{a short exact sequence } \, R' \rightarrowtail Q' \twoheadrightarrow X \, \text{ with} \, R' \in \tilclass{R} \, , Q' \in \tilclass{Q} \,\}.
          \end{align*}
Moreover, $\class{W}$ is unique in the sense that if $\class{V}$ is another thick class for which $(\class{Q},\class{V},\class{R})$ is a Hovey triple, then necessarily $\class{V} = \class{W}$.
\end{theorem}
Theorem~\ref{them-how to construct hovey triples} just appeared in~\cite{gillespie-hovey triples}. It provides a powerful method for constructing hereditary abelian model structures and is the main tool used in this paper.

\subsection{Gorenstein flat modules}\label{subsec-G-flat} We now recall the key definitions.

\begin{definition}
By a \textbf{complete flat resolution} we mean an exact chain complex $F$ of flat left $R$-modules
$$F = \cdots \xrightarrow{}  F_1  \xrightarrow{} F_0   \xrightarrow{}  F_{-1} \xrightarrow{} F_{-2}    \xrightarrow{} \cdots $$ for which $I \otimes_R F$ is also exact whenever $I$ is an injective right $R$-module.
\end{definition}

\begin{definition}
A left $R$-module $M$ is called \textbf{Gorenstein flat} if $M = Z_{-1}F$ (that is, $\ker{(F_{-1} \xrightarrow{} F_{-2})})$ for some complete flat resolution $F$.
We let $\class{GF}$ denote the class of all Gorenstein flat modules and $\class{GC} = \rightperp{\class{GF}}$. We call the modules in $\class{GC}$ \textbf{Gorenstein cotorsion}.
\end{definition}

It is easy to see that any flat module is Gorenstein flat and consequently any Gorenstein cotorsion module is cotorsion.

It is well known that $(\class{F},\class{C})$, where $\class{F}$ is the class of flat modules and $\class{C} = \rightperp{\class{F}}$ is the class of cotorsion modules, is a complete hereditary cotorsion pair. This was shown in~\cite{enochs-flat-cover-theorem}, and is also proved in~\cite{enochs-jenda-book}.

\begin{definition}
Call an $R$-module $N$ \textbf{flat cotorsion} if it is both flat and cotorsion. That is, if it belongs to the \emph{core} $\class{F} \cap \class{C}$ of the flat cotorsion pair $(\class{F},\class{C})$.
\end{definition}

Holm shows in~\cite{holm-Goren} that for a right coherent ring $R$ the class $\class{GF}$ is projectively resolving. This statement is encompassed in the following important theorem.

\begin{theorem}[Enochs, Jenda, Lopez-Ramos~\cite{enochs-Goren flat covers}]\label{them-G-flat covers}
Let $R$ be a right coherent ring and $\class{GF}$ the class of Gorenstein flat left $R$-modules. Then $(\class{GF},\class{GC})$ is a complete hereditary cotorsion pair.
\end{theorem}

\section{The Gorenstein flat model structure}\label{sec-model struc}

Let $R$ be a right coherent ring and $R$-Mod the category of left $R$-modules. Our goal is to apply Theorem~\ref{them-how to construct hovey triples} to obtain a hereditary abelian model structure on $R$-Mod whose homotopy category is a generalization of the stable module category, Stmod($R$), in the case that $R$ is Gorenstein. We let $(\class{F},\class{C})$ and $(\class{GF},\class{GC})$ denote, respectively, the flat cotorsion pair and the Gorenstein flat cotorsion pair from Section~\ref{subsec-G-flat}.
Note that to prove existence of the model structure we only need to show that the two cotorsion pairs have the same core. This is accomplished by the following lemma and proposition. 


\begin{lemma}\label{lemma-Gorenstein cotorsion modules}
The following are equivalent for an $R$-module $N$.
\begin{enumerate}
\item $N$ is Gorenstein cotorsion.
\item $N$ is cotorsion and the complex $\Hom_R(F,N)$ is exact for all complete flat resolutions $F$.
\end{enumerate}
\end{lemma}

\begin{proof}
 For the proof we let $S^0(N)$ denote the chain complex consisting of $N$ in degree zero and 0 in all other degrees. Note that if $N$ is Gorenstein cotorsion then it is clearly cotorsion, so we will assume throughout the proof that $N$ is cotorsion.
Then for any complete flat resolution $F$, we see that $\Ext^1_{\ch}(F,S^0(N))$ coincides with the subgroup $\Ext^1_{dw}(F,S^0(N))$; this is the Yondeda Ext subgroup consisting of all degreewise split short exact sequences $$0 \xrightarrow{} S^0(N) \xrightarrow{} Z \xrightarrow{} F \xrightarrow{}0.$$
Now this subgroup vanishes for all complete flat resolutions $F$ if and only if any chain map $F \xrightarrow{} S^0(N)$ with $F$ a complete flat resolution is null homotopic; that is, if  $\Hom_R(F,N)$ is exact for all complete flat resolutions $F$. Summarizing, we have that $\Hom_R(F,N)$ is exact for all complete flat resolutions $F$ if an only if $\Ext^1_{\ch}(F,S^0(N)) = 0$ for all complete flat resolutions $F$. But finally, since any such $F$ is an exact complex, the following isomorphism holds by~\cite[Lemma~4.2]{gillespie-degreewise-model-strucs}: $$\Ext^1_{\ch}(F,S^0(N) \cong \Ext^1_R(Z_{-1}F,N).$$
Now the Lemma follows since the class of Gorenstein flat modules coincides with the class of all $Z_nF$ where $F$ is a complete flat resolution.
\end{proof}

\begin{proposition}\label{prop-flat cores equal}
Let $R$ be a right coherent ring. Then the flat cotorsion pair $(\class{F},\class{C})$ and the Gorenstein flat cotorsion pair $(\class{GF},\class{GC})$ have the same core. That is, $\class{GF} \cap \class{GC} = \class{F} \cap \class{C}$.
\end{proposition}

\begin{proof}
($\subseteq$) Say $N \in \class{GF} \cap \class{GC}$. Since $\class{GC} \subseteq \class{C}$ we only need to argue that $N$ is flat. We start by writing $N = Z_0F$ where $F$ is some complete flat resolution, and we factor the differential $d_1 : F_1 \xrightarrow{} F_0$ as $F_1 \xrightarrow{e} N \xrightarrow{i} F_0$. So $d_1  = ie$ where $e$ is an epimorphism and $i$ is a monomorphism. By Lemma~\ref{lemma-Gorenstein cotorsion modules}  we know that $\Hom_R(F,N)$ is exact. But since $e \in \ker{d^*_2}$, the exactness of $\Hom_R(F,N)$ gives us $e \in \ker{d^*_2} = \im{d^*_1}$. This means we have a map $p \in \Hom_R(F_0,N)$ such that $pd_1 = e$. So $pie = e$. So $pi = 1_N$. So $N$ is a retract of the flat $F_0$, and so $N$ too is flat.

Before proving the reverse inclusion ($\supseteq$), we recall that given a left (resp. right) $R$-module $M$, its {character module} is defined to be the right (resp. left) $R$-module $M^+ = \Hom_{\Z}(M,\Q)$. It is a standard fact that $M$ is flat if and only if $M^+$ is injective. A proof can be found in~\cite{enochs-jenda-book}, for example.

($\supseteq$) Now suppose $N$ is flat cotorsion. This time we see $N$ is clearly Gorenstein flat, so the point is to show that $N$ must be Gorenstein cotorsion too. We again use Lemma~\ref{lemma-Gorenstein cotorsion modules} above. That is, let $F$ be an arbitrary complete flat resolution. We will be done once we show $\Hom_R(F,N)$ is also exact. First, note that the double character dual $N^{++}$ does have the property that
$\Hom_R(F,N^{++})$ is exact. Indeed, $\Hom_R(F,N^{++}) = \Hom_R(F,\Hom_{\Z}(N^+,\Q)) \cong \Hom_{\Z}(N^+ \tensor_R F,\Q))$, and since $N$ is flat we know $N^+$ is an injective (right) $R$-module. So this last complex must be exact just by assuming that $N$ is flat.

Now that we have shown $\Hom_R(F,N^{++})$ must be exact whenever $N$ is flat, it is left to argue that $\Hom_R(F,N)$ must also be exact when $N$ is flat cotorsion. We will argue that $\Hom_R(F,N)$ is in fact a retract of $\Hom_R(F,N^{++})$; since exact complexes are closed under retracts this will complete the proof. Now $N$ flat implies $N^+$ is injective. In particular, $N^+$ is absolutely pure and since $R$ is (right) coherent we conclude from~\cite[Theorem~2.2]{fieldhouse} that $N^{++}$ is also flat.
 Then from the proof of~\cite[Proposition~5.3.9]{enochs-jenda-book} we see that there is a pure exact sequence $$0 \xrightarrow{} N \xrightarrow{} N^{++} \xrightarrow{} N^{++}/N \xrightarrow{} 0.$$ Since the class of flat modules is closed under pure quotients, we get that $N^{++}/N$ is also flat. But by the assumption that $N$ is also cotorsion, this means that the above sequence splits. Thus $N$ is a retract of $N^{++}$. Like all functors, $\Hom_R(F,-)$ must preserve retracts. So we conclude that $\Hom_R(F,N)$ must be an exact complex.
\end{proof}

\begin{theorem}\label{them-main}
Let $R$ be a right coherent ring. Then the category $R$-Mod of left $R$-modules has an abelian model structure, the \textbf{Gorenstein flat model structure}, as follows:
\begin{itemize}
\item The cofibrant objects coincide with the class $\class{GF}$ of Gorenstein flat modules.
\item The fibrant objects coincide with the class $\class{C}$ of cotorsion modules.
\item The trivially cofibrant objects coincide with the class $\class{F}$ of flat modules.
\item The trivially fibrant objects coincide with the class $\class{GC}$ of Gorenstein cotorsion modules.
\end{itemize}
An $R$-module $M$ fits into a short exact sequence $$0 \xrightarrow{} C \xrightarrow{} F \xrightarrow{} M \xrightarrow{} 0$$ with $F \in \class{F}$ and $C \in \class{GC}$ if and only if it fits into a short exact sequence $$0 \xrightarrow{} M \xrightarrow{} C' \xrightarrow{} F' \xrightarrow{} 0$$ with $F' \in \class{F}$ and $C' \in \class{GC}$. Modules $M$ with this property are precisely the trivial objects of the Gorenstein flat model structure.
\end{theorem}

\begin{proof}
As noted in Section~\ref{sec-prelims} we already know that $(\class{F},\class{C})$ is a complete hereditary cotorsion pair. Since $R$ is right coherent we also have from Theorem~\ref{them-G-flat covers} that $(\class{GF},\class{GC})$ is a complete hereditary cotorsion pair. Proposition~\ref{prop-flat cores equal} says that the cores of these cotorsion pairs are equal. Clearly $\class{F} \subseteq \class{GF}$, and this is equivalent to $\class{GC} \subseteq \class{C}$. So Theorem~\ref{them-how to construct hovey triples} produces a model structure exactly as described.
\end{proof}

Recall that a \textbf{Frobenius category} is an exact category with enough injectives and projectives and in which the projective and injective objects coincide. Given a Frobenius category $\class{A}$, we can form the stable category $\class{A}/\sim$, where $f \sim g$ iff $g-f$ factors through a projective-injective. The main fact about Frobenius categories is that the stable category has the structure of a triangulated category. See~\cite{happel-triangulated}. On the other hand, the homotopy category of an abelian model structure is naturally a pretriangulated category as well~\cite[Chapters~6 and~7]{hovey-model}, and equivalent to the stable category of a Frobenius category whenever it is hereditary by~\cite{gillespie-exact model structures}. In particular, we get the following corollary.

\begin{corollary}\label{cor-Frobenius}
Let $R$ be a right coherent ring. Then the full subcategory $\class{GF} \cap \class{C}$ of $R$-Mod consisting of the Gorenstein flat and cotorsion modules is a Frobenius category with respect to its inherited exact structure. The projective-injective objects are precisely the flat cotorsion modules. Moreover, the homotopy category of the Gorenstein flat model structure is triangle equivalent to the stable category $$(\class{GF} \cap \class{C})/\sim$$ where $f \sim g$ if and only if $g -f$ factors through a flat cotorsion module. 
\end{corollary}

\begin{proof}
By the inherited exact structure we mean that the short exact sequences (or conflations) are the usual short exact sequences but with all three terms in $\class{GF} \cap \class{C}$.  Since we have a hereditary Hovey triple $(\class{GF},\class{W},\class{C})$ whose core $\class{GF} \cap \class{W} \cap \class{C}$ equals the class of flat cotorsion modules, the result follows from~\cite[Sections~4 and~5]{gillespie-exact model structures}. See also~\cite[Section~4.2]{gillespie-hereditary-abelian-models}.
\end{proof}

We end by pointing out that the homotopy category of the Gorenstein flat model structure is equivalent to the usual stable module category Stmod($R$) in the case that  $R$ is a Ding-Chen ring in the sense of~\cite{gillespie-ding}. These are the coherent analogs of Gorenstein rings. In particular, the class of Ding-Chen rings includes Gorenstein rings and hence quasi-Frobenius rings.

\begin{corollary}\label{corollary}
Let $R$ be a Ding-Chen ring. That is, a (left and right) coherent ring which has finite FP-injective dimension as both a left and right module over itself. Then a module is trivial in the Gorenstein flat model structure if and only if it has finite flat dimension, or equivalently, if and only if it has finite FP-injective dimension. In the case that $R$ is Gorenstein this is equivalent to the module having finite injective dimension and also equivalent to it having finite projective dimension.
\end{corollary}

\begin{proof}
Apply~\cite[Theorem~4.10]{gillespie-ding}. This theorem says that when $R$ is a Ding-Chen ring, we have another Hovey triple $(\class{GF},\class{V},\class{C})$ where the class $\class{V}$ of trivial objects contains exactly the modules described in the corollary. It follows from the uniqueness property of Theorem~\ref{them-how to construct hovey triples} that the class $\class{V}$ of trivial objects must coincide with the class $\class{W}$ from Theorem~\ref{them-main}.
\end{proof}

Two other model structures for generalizing the stable module category of a ring appear in~\cite{bravo-gillespie-hovey}. In the case that $R$ is (right) coherent the first has as its cofibrant objects the \emph{Ding projective} modules from~\cite{gillespie-ding}. The second has as its fibrant objects the \emph{Ding injective} modules from~\cite{gillespie-ding}. When $R$ is Gorenstein, or even Ding-Chen, these two model structures also each have the same trivial objects as in Corollary~\ref{corollary}. Hence all three models produce equivalent homotopy categories in this case. But for a general (right) coherent ring, things are less clear. In particular, we don't know precisely when the Ding projective model and the Ding injective model capture equivalent homotopy categories. However, it is shown in~\cite[Section~5, Prop.~5.2]{estrada-gillespie-coherent schemes}, that for any (right) coherent ring $R$, the trivial objects of the Gorenstein flat model structure coincide with the trivial objects of the Ding projective model structure. The homotopy category of the latter was called the \emph{projective stable module category of $R$} in~\cite{bravo-gillespie-hovey}. Thus we can think of the model structure constructed in this paper as a ``flat model'' for this homotopy category. The paper~\cite{estrada-gillespie-coherent schemes} continues in this direction, considering the problem of extending the flat model structure to coherent schemes.


\providecommand{\bysame}{\leavevmode\hbox to3em{\hrulefill}\thinspace}
\providecommand{\MR}{\relax\ifhmode\unskip\space\fi MR }
\providecommand{\MRhref}[2]{%
  \href{http://www.ams.org/mathscinet-getitem?mr=#1}{#2}
}
\providecommand{\href}[2]{#2}

\end{document}